\documentclass[12pt,leqno]{article}

\usepackage{amssymb, graphicx, amscd, amsmath, amsthm, array, fancyvrb, relsize, paralist,url}

\newtheorem{thm}{Theorem}[section]
\newtheorem{lem}[thm]{Lemma}

\newtheorem{prop}[thm]{Proposition}

\numberwithin{equation}{section}
\def\Fo{\mathbb{F}}
\def\F{\Fo_p}
\def\fx4{\frac{x}{4}}
\def\P{\mathcal{P}}
\def\M{\mathcal{M}}
\def\N{\mathcal{N}}
\def\L{\mathcal{L}}
\def\U{\mathcal{U}}
\def\V{\mathcal{V}}
\def\W{\mathcal{W}}
\def\A{\mathcal{A}}
\def\B{\mathcal{B}}

\def\C{\mathcal{C}}
\def\D{\mathcal{D}}
\def\S{\mathcal{S}}
\def\3F2{{}_3\hspace{-1pt}F_2}
\def\2F1{{}_2\hspace{-1pt}F_1}
\def\h2F1{{}_2\hspace{-1pt}\widehat{F}_1}
\def\Fq2{\F_{q^2}}

\def\c4{\chi_4}
\def\oc4{\overline{\chi_4}}

\def\l({\left(}
\def\r){\right)}

\def\dim{\mathrm{dim \ }}

\def\GL{\mathrm{GL }}
\title{Classification of certain types of maximal  matrix subalgebras}

\author{\\ \\ John Eggers\\
Department of Mathematics\\
University of California at San Diego\\
La Jolla, CA  92093-0112 \\
jeggers@ucsd.edu
\\ \\
Ron Evans\\
Department of Mathematics\\
University of California at San Diego\\
La Jolla, CA  92093-0112 \\
revans@ucsd.edu
 \\ \\
Mark Van Veen\\
Varasco LLC \\ 
2138 Edinburg Avenue\\ 
Cardiff by the Sea, CA 92007  \\
mark@varasco.com
}

\date{}
\begin{document}

\maketitle



\begin{abstract}
Let $\M_n(K)$ denote the algebra of $n \times n$ matrices over a field $K$
of characteristic zero.  A nonunital subalgebra $\N \subset \M_n(K)$
will be called a {\em nonunital intersection} if $\N$ is the intersection
of two unital subalgebras of $\M_n(K)$.   Appealing to recent work of Agore,
we show that for  $n \ge 3$, the dimension (over $K$) of a nonunital
intersection is at most $(n-1)(n-2)$, and we completely classify the
nonunital intersections of maximum dimension $(n-1)(n-2)$.
We also classify the unital subalgebras of maximum dimension
properly contained in a parabolic subalgebra of 
maximum dimension in $\M_n(K)$.

\end{abstract}

\maketitle

\section{Introduction}
Let $\M_n(F)$ denote the algebra of $n \times n$ matrices over a field $F$.
For some interesting sets $\Lambda$ of subspaces $\S \subset \M_n(F)$, 
those $\S \in \Lambda$ of maximum dimension over $F$ have been completely 
classified.
For example, a theorem of Gerstenhaber and Serezhkin \cite[Theorem 1]{Pazzis} 
states that when $\Lambda$ is the 
set of subspaces $\S \subset \M_n(F)$ 
for which every matrix in $\S$ is nilpotent,
then each $\S \in \Lambda$ of maximum dimension is
conjugate to the algebra of all strictly upper triangular matrices   
in $\M_n(F)$.
For another example,  it is shown in \cite[Prop. 2.5]{Agore} 
that when $\Lambda$ is the set of proper unital subalgebras 
$\S \subset \M_n(F)$ and $F$ is an algebraically closed field of characteristic 
zero, then each $\S \in \Lambda$ of maximum dimension is a
parabolic subalgebra of maximum dimension in $\M_n(F)$. 

The goal of this paper is to classify the elements in 
$\Lambda$ of maximum dimension in the cases
$\Lambda = \Gamma$ and $\Lambda = \Omega$, where the sets
$\Gamma$ and $\Omega$ are defined below.

In Isaac's text \cite[p. 161]{Isaacs}, every ring is required to have a unity,
but the unity in a subring need not be the same as the unity in 
its parent ring.  Under this definition, a ring may have subrings whose
intersection is not a subring.  This motivated us to study examples
of pairs of unital subrings in $\M_n(K)$ whose intersection $\N$ is nonunital,
where $K$ is a field of characteristic zero.
We call such $\N$ a {\em nonunital intersection}
and we let $\Gamma$ denote the set of all nonunital intersections
$\N \subset \M_n(K)$.  Note that $\Gamma$ is closed under transposition and 
conjugation, i.e., if $\N \in \Gamma$, then $\N^{\,\rm {T}} \in \Gamma$
and $S^{-1} \N S \in \Gamma$ for any invertible $S \in \M_n(K)$. 

In order to define $\Omega$, we need to establish some notation.
For brevity, write $\M = \M_n = \M_n(K)$.
In the spirit of \cite [p. viii]{GW}, 
we define a subalgebra of $\M$ to be a vector subspace of $\M$ over $K$
closed under the multiplication of $\M$ (cf. \cite[p. 2]{GW}); thus
a subalgebra need not have a unity, and the unity of a unital subalgebra
need not be a unity of the parent algebra.
Subalgebras $\A, \ \B \subset \M$ are said to be similar if
$\A = \{S^{-1} B S:  B \in \B\}$ for some invertible $S \in \M$.
The notation $\M[R_n]$ will be used for the subalgebra of $\M$
consisting of those matrices whose $n$-th row is zero.   Similarly,
$\M[R_n,C_n]$ indicates that the $n$-th row and $n$-th column are zero, etc.
For $1 \le i, j \le n$, let $E_{i,j}$ denote the elementary matrix 
in $\M$ with a single entry 1 in row $i$,
column $j$, and $0$ in each of the other $n^2-1$ positions.
The identity matrix in $\M$ will be denoted by $I$.
For the maximal parabolic subalgebra 
$\P:=\M[R_n] + K E_{n,n}$ in $\M$, define
$\Omega$ to be the set of proper subalgebras $\B$ of  $\P$
with $\B \ne \M[R_n]$.

We now describe Theorems 3.1--3.3, our main results.
Theorem 3.1 shows that $\dim\N \le (n-1)(n-2)$ for each $\N \in \Gamma$.
Theorem 3.2 shows that up to similarity, $\W: = \M[R_n,R_{n-1},C_n]$ 
and $\W ^ {\rm {T}}: = \M[R_n,C_{n-1},C_n]$ are the only subalgebras
in $\Gamma$ having maximum  dimension $(n-1)(n-2)$. 
In Theorem 3.3,  
we show that $\dim \B \le n^2 -2n +3$ for each $\B \in \Omega$,
and we classify all $\B \in \Omega$ of maximum dimension $n^2 -2n +3$.

The proofs of our theorems depend on four lemmas which are proved
in Section 2.  Lemma 2.1 shows that $\W$  (and hence also $\W ^ {\rm {T}}$)
is a nonunital intersection of dimension $(n-1)(n-2)$ when $n \ge 3$.
Lemmas 2.2 and 2.3 show that $\dim \L \le n(n-1)$ for any nonunital
subalgebra $\L \subset \M$, and when equality holds, $\L$ must be similar
to $\M[R_n]$ or $\M[C_n]$.  
(Thus if $\Lambda$ denotes the set of nonunital subalgebras $\L \subset \M$,
Lemmas 2.2 and 2.3 classify those $\L \in \Lambda$ of maximum dimension.)
Lemma 2.4 shows that if $\U \subset \M$
is a subalgebra with unity different from $I$, then some conjugate of $\U$
is contained in $\M[R_n,C_n]$.

\section{Lemmas}

Recall the definition $\W: = \M[R_n,R_{n-1},C_n]$.

%Lemma 2.1%
\begin{lem}
For $n \ge 3$, $\W \in \Gamma$ and $\ \dim \W = (n-1)(n-2)$.
\end{lem}

\begin{proof}
For $n>1$, define $A \in \M$
by $A=I + E_{n,n-1} \ $.
Note that $A^{-1} = I - E_{n,n-1} \ $.
A straightforward computation shows that for $M \in \M[R_n, C_n]$,
the conjugate $AMA^{-1}$ is obtained from $M$ by replacing the (zero)
bottom row of $M$ by the $(n-1)$-th row of $M$.  Since 
the bottom two rows of $AMA^{-1}$ are identical, it follows that
\[
AMA^{-1} \in \M[R_n, C_n] \cap A\M[R_n, C_n]A^{-1}
\mbox{  if and only if  } AMA^{-1} \in \W.
\]
Since $\W=A^{-1}\W A$, this shows that $\W$ is the intersection
of the unital subalgebras $A^{-1}\M[R_n, C_n]A$ and $\M[R_n, C_n]$.
To see that $\W$ is nonunital, note that $E_{1,n-1}$ is a nonzero
matrix in $\W$ for which $E_{1,n-1}W$ is the zero matrix for each
$W \in \W$;   thus $\W$ cannot have a right identity, so 
$\W \in \Gamma$.
Finally, it follows from the definition of $\W$ that
$\ \dim \W = (n-1)(n-2)$.
\end{proof}

\noindent {\em Remark: }
The same proof shows that $\W \in \Gamma$ holds when the field $K$
is replaced by an arbitrary ring $R$ with $1 \ne 0$. 
If moreover $R$ happens to be
commutative, then the dimension of the algebra $\W$ over $R$ 
is well-defined
\cite[p. 483]{Rotman} and it equals $(n-1)(n-2)$.

%Lemma 2.2%
\begin{lem}
For any nonunital subalgebra $\L \subset \M$, $\ \dim \L \le n(n-1)$.
\end{lem}

\begin{proof}
It cannot happen that $\L +KI =\M$, otherwise $\L$ would be a two-sided
proper ideal of $\M$, contradicting the fact that $\M$ is a simple ring
\cite[p. 280]{Rotman}.
Since $\L +KI$ is a proper subalgebra of $\M$ containing the unity $I$,
it follows from Agore \cite[Cor. 2.6]{Agore} that
$\ \dim \L = -1 + \ \dim (\L + KI) \le n(n-1)$.
\end{proof}

%Lemma 2.3%
\begin{lem}
Any nonunital subalgebra $\L \subset \M$ with $\ \dim \L = n(n-1)$
must be similar to either $\M[R_n]$ or $\M[C_n] = \M[R_n]^{\rm {T}}$.
\end{lem}

\begin{proof}
Consider the two parabolic subalgebras $\P, \P' \subset \M$
defined by
\[
\P = \P_K = \M[R_n] + K E_{n,n} \ ,  \quad \P' = \P'_K = \M[C_1] + K E_{1,1} \ .
\]
Note that $\P'$ is similar to the transpose $\P^{\rm {T}}$.
Since $\L +KI$ is a proper subalgebra of $\M$ of dimension $n(n-1)+1$,
it follows from  Agore \cite[Prop. 2.5]{Agore} that $\L +KI$ is similar to
$\P$ or $\P'$, under the condition that $K$ is algebraically closed.
However, Nolan Wallach \cite{Wallach} has proved that this condition
can be dropped; see the Appendix.   Thus, replacing $\L$ by a conjugate if
necessary, we may assume that $\L +KI =\P$ or $\L +KI =\P^{\rm {T}}$.
We will assume that $\L +KI =\P$, since the proof for $\P^{\rm {T}}$
is essentially the same.  It suffices to show that $\L$ is similar
to $\M[R_n]$ or $\M[C_1]$, since $\M[C_1]$ is similar to $\M[C_n]$.

Assume temporarily
that each $L \in \L$ has all entries 0 in its upper left $(n-1) \times
(n-1)$ corner.   Then $n=2$, because if $n \ge 3$, then every matrix
in $\P$ would have a zero entry in row 1, column 2, contradicting
the definition of $\P$. Since $\L \subset \M_2[C_1]$ and both sides
have dimension 2, we have $\L = \M_2[C_1]$, which proves the theorem
under our temporary assumption.

When the temporary assumption is false, there exists
$L \in \L$ with the entry 1 in row $i$, column $j$ for some fixed pair
$i,j$ with $1 \le i,j \le n-1$.
Since $E_{i,i}$ and $E_{j,j}$ are in $\P = \L +KI$ and $\L$ is a 
two-sided ideal of $\P$, we have $E_{i,j} =E_{i,i} L E_{j,j} \in \L$.
Consequently, $E_{a,b} =E_{a,i} E_{i,j} E_{j,b} \in \L$
for all pairs $a,b$ with $1 \le a \le n-1$ and $1 \le b \le n$.
Therefore
\[
\M[R_n] = \sum_{a=1}^{n-1} \sum_{b=1}^{n} K E_{a,b} \subset \L,
\]
and since both $\M[R_n]$ and $\L$ have the same dimension $n(n-1)$,
we conclude that $\L = \M[R_n]$.
\end{proof}

\noindent {\em Remark: }
Any subalgebra $\B \subset \M$ properly containing $\M[R_n]$
must also contain $I$.  To see this, note that $\B$ contains
a nonzero matrix of the form
\[
B:= \sum_{i=1}^{n} c_i E_{n,i} \ , \quad  c_i \in K.
\]
If $c_j = 0$ for all $j < n$, then $E_{n,n} \in \B$, so $I \in \B$.
On the other hand, if $c_j \ne 0$ for some $j < n$, then
$E_{n,n}=c_j^{-1 }B E_{j,n} \in \B$, so again $I \in \B$.

%Lemma 2.4%
\begin{lem}
Suppose that a subalgebra  $\U \subset \M$ has a unity $e \ne I$.
Then $S^{-1} \U S \subset \M[R_n,C_n]$ for some invertible $S \in \M$.
\end{lem}

\begin{proof}
Let $r$ be the rank of the matrix $e$.
Note that $e$ is idempotent, so by \cite[p. 27]{OCV},
there exists an invertible $S \in \M$ for which $S^{-1} e S = D_r$,
where $D_r$ is a diagonal matrix with entries 1 in rows 1 through $r$,
and entries 0 elsewhere.  Replacing $\U$ by
$S^{-1} \U S$ if necessary, we may
assume that $e=D_r$.  Since $r \le n-1$, we have
\[
\U = e \, \U e \subset e \M e = D_r \M D_r \subset D_{n-1}\M D_{n-1} 
= \M[R_n, C_n].
\]
\end{proof}

\section{Theorems}

Recall that $\Gamma$ is the set of all nonunital intersections in
$\M$.

%Theorem 3.1%
\begin{thm}
If $\N \in \Gamma$, then $\ \dim \N \le (n-1)(n-2)$.
\end{thm}

\begin{proof}
Let $\N \in \Gamma$, so that $\N = \U \cap \V$
for some pair of unital subalgebras $\U, \V \subset \M$.
Since $\N$ is nonunital, one of $\U, \V$, say $\U$, does not
contain $I$.  Thus $\U$ contains a unity $e \ne I$.  
Define $S$ as in Lemma 2.4.
Replacing $\U$, $\V$, $\N$ by 
$S^{-1} \U S$, $S^{-1} \V S$, $S^{-1} \N S$, if necessary, we 
deduce from Lemma 2.4 that $\U$ is contained in $\M[R_n,C_n]$.
Since $\N$ is a nonunital subalgebra of $\U \subset \M[R_n, C_n]$,
it follows from Lemma 2.2 with $(n-1)$ in place of $n$ that
$\ \dim \N \le (n-1)(n-2)$.
\end{proof}

%Theorem 3.2%
\begin{thm}
Let $n \ge 3$.  Then up to similarity, $\W$ and $\W^{\rm {T}}$ are the 
only subalgebras of $\M$ in $\Gamma$ having dimension $(n-1)(n-2)$.
\end{thm}

\begin{proof}
By Lemma 2.1, every subalgebra of $\M$ similar to $\W$ or $\W^{\rm {T}}$
lies in $\Gamma$ and has dimension $(n-1)(n-2)$.
Conversely, let $\N \in \Gamma$ with $\dim \ \N = (n-1)(n-2)$.
We must show that $\N$ is similar to $\W$ or $\W^{\rm {T}}$.

We may assume, as in the proof of Theorem 3.1, that
$\N$ is a nonunital subalgebra of $\M[R_n, C_n]$.
Let $\L$ be the subalgebra of $\M_{n-1}$ consisting of those
matrices in the upper left
$(n-1) \times (n-1)$ corners of the matrices in $\N$.
Since $\dim \ \L = \dim \ \N = (n-1)(n-2)$, it follows from
Lemma 2.3 that $\L$ is similar to $\M_{n-1}[R_{n-1}]$ or
$\M_{n-1}[C_{n-1}]$.   
Thus $\N$ is similar to
$\W = \M[R_n, R_{n-1}, C_n]$ 
or $\W^{\rm {T}} = \M[R_n, C_{n-1}, C_n]$.
\end{proof}

Recall that  $\Omega$ denotes the set of proper subalgebras 
$\B \ne \M[R_n]$ in $\P$.

%Theorem 3.3%
\begin{thm}
Let $\B \in \Omega$.
Then $\dim \B \le n^2 -2n +3$.
If $\B$ has maximum dimension $n^2 -2n +3$, then $\B$ is similar to one of
\[
\M E_{n,n} +  \M[R_n,C_1] + K E_{1,1}, \quad \quad 
\M E_{n,n} +  \M[R_n,R_{n-1}] + K E_{n-1,n-1}.
\]
\end{thm}

\begin{proof}
Let $e \in \M$ denote the diagonal matrix of rank $n-1$ with 
entry $0$ in row $n$ and entries $1$ in the remaining rows.
Because $e$ is a left identity in $\M[R_n]$
and $\B e \subset \M[R_n,C_n]$, it follows that $\B e$ is an algebra.

First suppose that $\B e = \M[R_n,C_n]$.   Then $\P = \C + \D$, where
\[
\C = \B + K E_{n,n}, \quad \D = \M[R_n] E_{n,n}.
\]
We proceed to show that
$\C \cap \D$ is zero.
Assume for the purpose of contradiction that there exists a nonzero matrix
$B \in \C \cap \D$.   Then $B \in \B$.   We have
$\B B  = \D$, since the matrices in $\B$ have all possible
submatrices in their upper left $(n-1)$ by $(n-1)$ corners.
Thus $\D \subset \B \subset \C$, 
which implies that $\M[R_n] \subset \B$ and
$\P = \C = \B + K E_{n,n}$.
If $K E_{n,n} \subset \B$, then $\B = \P$, 
and if $K E_{n,n}$ is not contained in $\B$, then
$\B = \M[R_n]$; either case contradicts the fact that
$\B \in \Omega$. 

Since $\C \cap \D$ is zero, 
\[
\dim \B \le \dim \C = \dim \P - \dim \D = (n^2-n+1)-(n-1) = n^2-2n+2.
\]
Thus 
\[
\dim \B < n^2 - 2n +3,
\]
so the desired upper bound for $\dim \B$  holds 
when $\B e = \M[R_n,C_n]$.

Next suppose that $\B e$ is a proper subalgebra of $\M[R_n,C_n]$.
We proceed to show that
\[
d:= \dim \B e \le (n-1)(n-2) +1,
\]
by showing that
\[
\dim \L \le (n-1)(n-2) +1 
\]
for every proper subalgebra $\L$
of $\M_{n-1}$. If $\L$ is nonunital,
then 
\[
\dim \L \le (n-1)(n-2) < (n-1)(n-2)+1 
\]
by Lemma 2.2 (with $n-1$ in place of $n$).
If $\L$ contains a unit different from the identity of $\M_{n-1}$, then by
Lemma 2.4 (with $\L$ in place of $\U$),
\[
\dim \L \le  \dim \M[R_{n-1}, C_{n-1}]=(n-2)^2 < (n-1)(n-2) +1.
\]
If $\L$ contains the identity of $\M_{n-1}$, then by Agore
\cite[Cor. 2.6]{Agore},
\[
\dim \L \le (n-1)(n-2) +1.
\]
This completes the demonstration that $d \le (n-1)(n-2)+1$.

Let $B_1e, B_2e, \dots, B_de$ be a basis for $\B e$, with $B_i \in \B$.
Since $\B$ is a subspace of the vector space spanned by the $d+n$
matrices
\[
B_1, \dots, B_d, E_{1,n}, \dots, E_{n,n},
\]
it follows that
\[
\dim \B \le d+n \le (n-1)(n-2)+1 +n = n^2 -2n +3.
\]
Thus the desired upper bound for $\dim \B$ holds in all cases.

The argument above shows that when we have the equality
\[
\dim \B = d+n = (n-1)(n-2)+1 +n = n^2 -2n +3,
\]
then
\[
\B = \B e + \M E_{n,n}.
\]
Moreover, from the equality $d = \dim \B e = (n-1)(n-2)+1$,
it follows from Agore \cite[Prop. 2.5]{Agore}
(again appealing to the Appendix to dispense with the condition
of algebraic closure) that 
there is an invertible matrix $S$ in the set $\M[R_n, C_n] + E_{n,n}$
such that $S^{-1} \B e S$ is equal to one of
\[
\M[R_n,C_n,C_1] + K E_{1,1}, \quad \quad
\M[R_n,C_n,R_{n-1}] + K E_{n-1,n-1}.
\]
Since $ S^{-1} \M E_{n,n} S = \M E_{n,n}$,
we achieve the desired classification of $\Omega$.
\end{proof}

\section{Appendix}

Let $F$ be a field of characteristic $0$ with 
algebraic closure $\overline{F}$.
Given a proper subalgebra $\C \subset \M_n(F)$ of maximum dimension,
Agore \cite[Prop. 2.5, Cor. 2.6]{Agore} proved that the $\bar{F}$-span
of $\C$ is similar over $\bar{F}$ to the $\bar{F}$-span of $\D$
for some parabolic subalgebra $\D$ of maximum dimension in $\M_n(F)$.
The purpose of this Appendix is to deduce that $\C$ is similar 
over $F$ to $\D$.

\begin{lem} {\rm (Wallach)}
Let $\A$ be a subspace of $\M_{n}(F)$ of dimension $n-1$ such that 
$\A\otimes _{F}\overline{F}$ has basis of one of the following two forms:

a) $x_{1}\otimes\lambda_{1},x_{2}\otimes\lambda_{1},...,x_{n-1}\otimes
\lambda_{1}$, with $\lambda_{1}\in(\bar{F}^{n})^{\ast},x_{j}\in\bar{F}^{n}$
and $\lambda_{1}(x_{j})=0$,

b) $x_{1}\otimes\lambda_{1},x_{1}\otimes\lambda_{2},...,x_{1}\otimes
\lambda_{n-1}$, with $\lambda_{j}\in(\bar{F}^{n})^{\ast},x_{1}\in\bar{F}^{n}$
and $\lambda_{j}(x_{1})=0.$
Then in case a) $\A$ is $F-$conjugate (i.e. under $\GL(n,F)$) to the
span of the matrices $E_{i,n}$ with $i=1,...,n-1$, and in case b) $\A$ is
$F-$conjugate to the span of the matrices $E_{n,i}$ with $i=1,...,n-1$.
\end{lem}

\begin{proof}
In either case, if $X,Y\in \A$ then $XY=0$ and $X$ has rank $1$.
For $X$ of rank $1$,
we have $XF^{n}=Fy$ for some $y\neq0$. Thus there exists $\mu\in\left(
F^{n}\right)  ^{\ast}$ with $Xz=\mu(z)y=\left(  y\otimes\mu\right)  (z)$ for
all $z$. We conclude that $\A$ has a basis over $F$ of the form 
$X_{i}=y_{i} \otimes\mu_{i}$ for $i=1,...,n-1$. 

We now assume that case a) is true (the argument for the other case is
essentially the same). In case a), there exists $z\in$ $\bar{F}^{n}$ such
that 
\[
\left\{  X_{1}(z),...,X_{n-1}(z)\right\}   
\]
is linearly independent over
$\bar{F}$. This implies that
\[
\mu_{1}(z)\cdots\mu_{n-1}(z)y_{1}\wedge\cdots\wedge y_{n-1}\neq0\text{. }
\]
Thus $y_{1},...,y_{n-1}$ are linearly independent. But $0=X_{i}X_{j}%
=\mu_{i}(y_{j})y_{i}\otimes\mu_{j}$. Thus $\mu_{i}(y_{j})=0$ for
all $j=1,...,n-1$. Let $\nu$ be a non-zero element of $\left(  F^{n}\right)
^{\ast}$ such that $\nu(y_{i})=0$ for all $i=1,...,n-1.$ Then $\nu$ is unique
up to non-zero scalar multiple. Thus $y_{i}\otimes\nu$, $i=1,...,n-1 $ is an
$F$--basis of $\A$. Clearly there exists $g\in \GL(n,F)$ such that if
$e_{1},...,e_{n}$ is the standard basis and $\xi_{1},...,\xi_{n}$ is the dual
basis then $gy_{i}=e_{i}$ and $\nu\circ g=\xi_{n}$. This completes the proof
in case a).
\end{proof}

\begin{prop}{\rm (Wallach)}
Suppose that $\L\subset \M_{n}(F)$ is a subalgebra such that
$\L\otimes_{F}\overline{F}$ \ is either:

a) conjugate to the parabolic subalgebra $\P_{\bar{F}}$,

b) conjugate to the parabolic subalgebra
 $\left(  \P_{\bar{F}}\right)  ^{\rm {T}}.$

\noindent In case a) $\L$ is $F$--conjugate to $\P_{F}$. 
In case b) $\L$ is $F$--conjugate
to $\P_{F}^{\rm {T}}$.
\end{prop}

\begin{proof}
We just do case a) as case b) is proved in the same way. 
We look upon $\L$ as a Lie
algebra over $F$. Then Levi's theorem \cite[p. 91]{Jacobson} implies
that $\L=S\oplus R$ with $S$ a semi-simple Lie algebra and $R$ the radical (the
maximal solvable ideal). Thus $\L\otimes_{F}\overline{F}$ $=S\otimes
_{F}\overline{F}\oplus R\otimes_{F}\overline{F}$. Therefore $R\otimes_{F}
\overline{F}$ is the radical of $\L\otimes_{F}\overline{F}$. If we conjugate
$\L\otimes_{F}\overline{F}$ to $\P_{\bar{F}}$ via $h \in \GL(n, \bar{F})$, 
then we see that
\[
h[R\otimes_{F}\overline{F},R\otimes_{F}\overline{F}]h^{-1}
\]
has basis $E_{i,n} \ ,\ i=1,...,n-1.$ Thus hypothesis a) of Lemma 4.1  is
satisfied for $\A=[R,R]$. There exists therefore $g \in \GL(n,F)$ such that
$g \A g^{-1}$ has basis $E_{i,n} \ , \  i=1,...,n-1$. Assume that we have
replaced $\L$ with $g \L g^{-1}$. Then $\A$ has basis $E_{i,n} \ , \ 
i=1,...,n-1$. Since $[\L,\A]\subset \A$ and $\P_{F}$ is exactly the set of
elements $X$ of $\M_{n}(F)$ such that $[X,\A]\subset \A$, we have $\L\subset
\P_{F}$. Thus $\L = \P_{F}$, as both sides have the same dimension.
\end{proof}

\section*{Acknowledgment}
We are very grateful to Nolan Wallach for providing the Appendix
and for many helpful suggestions.

\end{document}